\documentclass[11pt,a4paper]{amsart}
\usepackage[utf8]{inputenc}
\usepackage[T1]{fontenc}
\usepackage[english]{babel}
\usepackage{lmodern}
\usepackage[oneside,DIV=12]{typearea}

\numberwithin{equation}{section}
\usepackage{mathtools}
\usepackage{hyperref}
\usepackage{verbatim}
\usepackage{pgf,tikz}

\usepackage[backend=bibtex,style=numeric,sorting=none]{biblatex}
\title{Totally Odd Depth-graded Multiple Zeta Values and Period Polynomials}
\author{Charlotte Dietze}
\address{Charlotte Dietze, Ludwig-Maximilians-Universit\"at M\"unchen, Mathematisches Institut, Theresienstr. 39, 80333 M\"unchen}
\email{dietze@mathematik.uni-muenchen.de}

\author{Chokri Manai}
\address{Chokri Manai, TU M\"unchen, Boltzmannstra\"se 3/III, 85748 Garching, Germany}
\email{chokri.manai@tum.de}

\author{Christian N\"obel}
\address{Christian N\"obel, Department of Mathematics ETH Z\"urich, Institute for Operations Research, Rämistrasse 101, 8092 Z\"urich, Switzerland}
\email{christian.noebel@math.ethz.ch}

\author{Ferdinand Wagner}
\address{Ferdinand Wagner, Max-Planck-Institut f\"ur Mathematik,Vivatsgasse 7,53111 Bonn, Germany}
\email{ferdinand.wagner@uni-bonn.de}
\date{September 2016} 

\newtheorem{thm}{Theorem}[section]
\newtheorem{lem}[thm]{Lemma}
\newtheorem{cor}[thm]{Corollary}
\newtheorem{con}[thm]{Conjecture}

\theoremstyle{definition}
\newtheorem{dfn}[thm]{Definition}

\theoremstyle{remark}
\newtheorem*{rem}{Remark}

\newcommand{\ihara}{\mathbin{\underline{\circ}}}
\newcommand{\faul}[2]{{\textstyle\binom{#1}{#2}}}
\newcommand{\diag}{\operatorname{diag}}

\newcommand{\rank}{\operatorname{rank}}
\newcommand{\Vect}{\mathsf{Vect}}
\newcommand{\im}{\operatorname{Im}}
\newcommand{\tr}[1]{\prescript{t\!}{}{#1}}
\renewcommand{\;}{\,}

\newcommand{\id}{\operatorname{id}}
\renewcommand{\phi}[2][r]{\varphi^{(#1)}_{#2}}
\newcommand{\phii}[3][r]{\varphi^{(#1)}_{#2}\big|_{\mathbf W_{#3}}}
\newcommand{\odd}{{\operatorname{odd}}}
\mathtoolsset{centercolon}

\begin{document}

\begin{abstract}
We study the relations between totally odd, motivic depth-graded multiple zeta values. Our main objective is to determine the rank of the matrix $C_{N,r}$ defined by Brown \cite{Brown}. We will give conditional proofs for (conjecturally optimal) upper bounds on $\rank C_{N,3}$ and $\rank C_{N,4}$, simplifying some arguments of Tasaka \cite{tasakaPublished}. Finally, we present a recursive approach to the general problem and identify a conjecture which would imply that $\rank C_{N,r}$ has the expected value.
\end{abstract}

\maketitle

\section{Introduction}
\subsection{The Broadhurst--Kreimer conjecture}
In this paper we will study $\mathbb{Q}$-linear relations among totally odd depth-graded multiple zeta values (MZVs), for which there conjecturally is a bijection with the kernel of a specific matrix $C_{N,r}$ connected to restricted even period polynomials (for a definition, see \cite{schneps} or \cite[Section 5]{gkz2006}).

For integers $n_1,\ldots,n_{r-1}\geq1$ and $n_r\geq2$, the MZV of $n_1,\ldots,n_r$ is defined as the number
\begin{equation*}
    \zeta(n_1,\ldots,n_r)\coloneqq \sum_{0<k_1<\cdots<k_r} \frac{1}{k_1^{n_1}\cdots k_r^{n_r}}\;.
\end{equation*}
We call the sum $n_1+\cdots+n_r$ of arguments the weight and their number $r$ the depth of $\zeta(n_1,\ldots,n_r)$. One classical question about MZVs is counting the number of linearly independent $\mathbb Q$-linear relations between MZVs. It is highly expected, but for now seemingly out of reach that there are no relations between MZVs of different weight. Such questions become reachable when considered in the motivic setting. Motivic MZVs $\zeta^{\mathfrak m}(n_1,\ldots,n_r)$ are elements of a certain $\mathbb Q$-algebra $\mathcal H=\bigoplus_{N\geq 0}\mathcal H_N$ which was constructed by Brown in \cite{brownMixedMotives} and is graded by the weight $N$. Any relation fulfilled by motivic MZVs also holds for the corresponding MZVs via the period homomorphism $per\colon\mathcal H\to \mathbb R$. 

We further restrict to depth-graded MZVs: Let $\mathcal Z_{N,r}$ and $\mathcal H_{N,r}$ denote the $\mathbb Q$-vector space spanned by the real respectively motivic MZVs of weight~$N$ and depth~$r$ modulo MZVs of lower depth. The depth-graded MZV of $n_1,\ldots,n_r$, that is, the equivalence class of $\zeta(n_1,\ldots,n_r)$ in $\mathcal Z_{N,r}$, is denoted by $\zeta_{\mathfrak D}(n_1,\ldots,n_r)$. The elements of $\mathcal H_{N,r}$ are denoted $\zeta_{\mathfrak D}^{\mathfrak m}(n_1,\ldots,n_r)$ analogously. The dimension of $\mathcal Z_{N,r}$ is predicted by the Broadhurst--Kreimer Conjecture.

\begin{con}[Broadhurst--Kreimer]
The generating function of the dimension of the space $\mathcal Z_{N,r}$ is given by
\begin{equation*}
    \sum_{N,r\geq0}\dim_{\mathbb Q}\mathcal Z_{N,r}\cdot x^N y^r\overset?= \frac{1- \mathbb E(x)y}{1-\mathbb O(x)y+\mathbb S(x) y^2 - \mathbb S(x)y^4}\;,
\end{equation*}
where we denote $\mathbb E(x)\coloneqq \frac{x^2}{1-x^2}=x^2+x^4+x^6+\cdots$, $\mathbb O(x)\coloneqq \frac{x^3}{1-x^2}=x^3+x^5+x^7+\cdots$, and $\mathbb S(x)\coloneqq \frac{x^{12}}{(1-x^4)(1-x^6)}$.
\end{con}
\begin{rem}
It should be mentioned that $\mathbb S(x)=\sum_{n>0}\dim\mathcal S_n\cdot x^n$, where $\mathcal S_n$ denotes the space of cusp forms of weight~$n$, for which there is an isomorphism to the space of restricted even period polynomials of degree $n-2$ (defined in \cite{schneps} or \cite[Section 5]{gkz2006}).
\end{rem}

\subsection{Brown's matrix $C_{N,r}$}
In his paper \cite{Brown}, Brown considered the $\mathbb Q$-vector space $\mathcal Z_{N,r}^\odd$ (respectively $\mathcal H_{N,r}^\odd$) of totally odd (motivic) and depth-graded MZVs, that is, $\zeta_{\mathfrak D}(n_1,\ldots,n_r)$ (respectively $\zeta_{\mathfrak D}^{\mathfrak m}(n_1,\ldots,n_r)$) for $n_i\geq3$ odd, and linked them to a certain explicit matrix $C_{N,r}$, where $N=n_1+\cdots+n_r$ denotes the weight. In particular, he showed that any right annihilator $(a_{n_1,\ldots,n_r})_{(n_1,\ldots,n_r)\in S_{N,r}}$ of $C_{N,r}$ induces a relation 
\begin{equation*}
    \sum_{(n_1,\ldots,n_r)\in S_{N,r}}a_{n_1,\ldots,n_r}\zeta_{\mathfrak D}^{\mathfrak m}(n_1,\ldots,n_r)=0\;,\text{ hence also }\sum_{(n_1,\ldots,n_r)\in S_{N,r}}a_{n_1,\ldots,n_r}\zeta_{\mathfrak D}(n_1,\ldots,n_r)=0
\end{equation*}(see Section~\ref{sec:preliminaries} for the notations) and conjecturally all relations in $\mathcal Z_{N,r}^\odd$ arise in this way. This led to the following conjecture (the uneven part of the Broadhurst--Kreimer Conjecture).
\begin{con}[Brown {\cite{Brown}}]\label{con:Brown}
The generating series of the dimension of $\mathcal Z_{N,r}^\odd$ and the rank of $C_{N,r}$ are given by
\begin{equation*}
    1+\sum_{N,r>0}\rank C_{N,r}\cdot x^Ny^r\overset?=1+\sum_{N,r>0}\dim_{\mathbb Q}\mathcal Z_{N,r}^\odd\cdot x^Ny^r\overset?=\frac1{1-\mathbb O(x)y+\mathbb S(x)y^2}\;.
\end{equation*}
\end{con}

\subsection{Summary of this paper}
The contents of this paper are as follows. In Section~\ref{sec:preliminaries}, we explain our notations and define the matrices $C_{N,r}$ due to Brown \cite{Brown} as well as $E_{N,r}$ and $E_{N,r}^{(j)}$ considered by Tasaka \cite{tasakaPublished}. In Section~\ref{sec:known results}, we briefly state some of Tasaka's results on the matrix $E_{N,r}$. Section~\ref{sec:main tools} is devoted to further investigate the connection between the left kernel of $E_{N,r}$ and restricted even period polynomials, which was first discovered by Baumard and Schneps \cite{schneps} and appears again in \cite[Theorem 3.6]{tasakaPublished}. In Section~\ref{sec:main results}, we will apply our methods to the cases $r=3$, $r=4$, and $r=5$. The first goal of Section~\ref{sec:main results} will be to show
\begin{thm}\label{thm:case3}\label{thm:case4}Assume that the map from Theorem~\ref{thm:injection} below is injective. We then have the following lower bounds (which Conjecture~\ref{con:Brown} predicts to be sharp):
\begin{align*}
    \sum_{N>0}\dim_{\mathbb Q}\ker C_{N,3}\cdot x^N&\geq 2 \mathbb O(x)\mathbb S(x)\;,\\
    \sum_{N>0}\dim_{\mathbb Q}\ker C_{N,4}\cdot x^N&\geq 3 \mathbb O(x)^2\mathbb S(x)-\mathbb S(x)^2\;.
\end{align*}
Here $\geq$ means that for every $N>0$ the coefficient of $x^N$ on the right-hand side does not exceed the corresponding one on the left-hand side.
\end{thm}
\begin{thm}\label{thm:case5}Assume that the map from Theorem~\ref{thm:injection} below is an isomorphism. We then have the following lower bound (which Conjecture~\ref{con:Brown} predicts to be sharp):
	\begin{equation*}
		\sum_{N>0}\dim_{\mathbb Q}\ker C_{N,5}\cdot x^N\geq 4 \mathbb O(x)^3\mathbb S(x)-3 \mathbb O(x)\mathbb S(x)^2\;.
	\end{equation*}
\end{thm}
\begin{rem}
	Theorem~\ref{thm:case3} was first announced in \cite{tasakaPublished}, without the condition that the map from Theorem~\ref{thm:injection} is injective. It has since become apparent that the original proof of injectivity given in \cite{tasakaPublished} contained a gap \cite{tasakaCorrigendum}. Assuming injectivity, our proof of Theorem~\ref{thm:case3} is a significant simplification of the arguments in \cite{tasakaPublished} since we only use tools from linear algebra and make no mention of the shuffle product.
\end{rem}
\begin{rem}
	After the first version of this paper had appeared, Tasaka \cite{tasakaFullProof} has given a full and unconditional proof of the upper bound
	\begin{equation}\label{eq:tabdnew}
		1+\sum_{N,r>0}\rank C_{N,r}\cdot x^Ny^r\leq \frac1{1-\mathbb O(x)y+\mathbb S(x)y^2}\,,
	\end{equation}
	which is predicted to be sharp by Conjecture~\ref{con:Brown}. In particular, this gives uncoditional proofs of Theorems~\ref{thm:case3} and~\ref{thm:case5}. However, the injectivity of the map in Theorem~\ref{thm:injection} is still open and is not implied by the unconditional result.
\end{rem}
In the last subsection, we present a recursive approach to determine the precise value of $\dim_{\mathbb Q}\ker C_{N,r}$. We identify and isomorphism conjecture due to Claire Glanois, see Conjecture~\ref{con:imiso}, which allows us to deduce the values for the values for $\dim_{\mathbb Q}\ker C_{N,r}$, as stated in Conjecture~\ref{con:Brown}.
\begin{thm}\label{thm:caser}
	If the map from Theorem~\ref{thm:injection} below  is injective and Conjecture~\ref{con:imiso} is true, then
	\begin{equation*}
		1+\sum_{N,r>0}\rank C_{N,r}\cdot x^Ny^r=\frac1{1-\mathbb O(x)y+\mathbb S(x)y^2}\,.	\end{equation*}
\end{thm}
In view of the upper bound \eqref{eq:tabdnew}, which was proved by Tasaka in \cite{tasakaFullProof}, we hope that our conditional Theorem~\ref{thm:caser} provides some progress towards obtaining the exact conjectured values using a recursive approach.

\subsection*{Acknowledgments}
This research was conducted as part of the Hospitanzprogramm (internship program) at the Max-Planck-Institut f\"ur Mathematik (Bonn). We are grateful for the invitation and the hospitality. We would like to express our deepest thanks to our mentor, Claire Glanois, for introducing us into the theory of multiple zeta values. Finally, we would like to thank Daniel Harrer, Matthias Paulsen, and J\"orn St\"ohler for many helpful comments.

\section{Preliminaries}\label{sec:preliminaries}
\subsection{Notations}
In this section we introduce our notations and we give some definitions. As usual, for a matrix $A$ we define $\ker A$ to be the set of right annihilators of $A$. Apart from this, we mostly follow the notations of Tasaka in his paper \cite{tasakaPublished}. Let
\begin{equation*}
    S_{N,r}\coloneqq \left\{(n_1,\ldots,n_r)\in\mathbb Z^r\ |\ n_1+\cdots+n_r=N,\ n_1,\ldots,n_r\geq3\text{ odd}\right\}\;,
\end{equation*}
where $N$ and $r$ are natural numbers. Since the elements of the set $S_{N,r}$ will be used as indices of matrices and vectors, we usually arrange them in lexicographically decreasing order. Let
\begin{equation*}
\mathbf V_{N,r}\coloneqq \left\langle \left.x_1^{m_1-1}\cdots x_r^{m_r-1}\ \right|\ (m_1,\ldots,m_r)\in S_{N,r}\right\rangle_\mathbb{Q}
\end{equation*}
denote the vector space of restricted totally even homogeneous polynomials of degree $N-r$ in $r$ variables. There is a natural isomorphism from $\mathbf V_{N,r}$ to the $\mathbb Q$-vector space $\Vect_{N,r}$ of $n$-tuples $(a_{n_1,\ldots,n_r})_{(n_1,\ldots,n_r)\in S_{N,r}}$ indexed by totally odd indices $(n_1,\ldots,n_r)\in S_{N,r}$, which we denote
\begin{align}\label{eq:natiso}
\begin{split}
    \pi\colon\mathbf V_{N,r}&\overset{\sim\,}{\longrightarrow}\Vect_{N,r}\\
    \sum_{(n_1,\ldots,n_r)\in S_{N,r}}a_{n_1,\ldots,n_r}x_1^{n_1-1}\cdots x_r^{n_r-1}&\longmapsto \left(a_{n_1,\ldots,n_r}\right)_{(n_1,\ldots,n_r)\in S_{N,r}}\;.
\end{split}    
\end{align}
We assume vectors to be row vectors by default.

Finally, let $\mathbf W_{N,r}$ be the vector subspace of $\mathbf V_{N,r}$ defined by
\begin{align*}
\mathbf W_{N,r}\coloneqq \left\{P\in\mathbf V_{N,r}\ |\ P(x_1,\ldots,x_r)\right.&=P(x_2-x_1,x_2,x_3,\ldots,x_r)\\&\left.\phantom=-P(x_2-x_1,x_1,x_3,\ldots,x_r)\right\}\;.
\end{align*}
That is, $P(x_1,x_2,x_3,\ldots,x_r)$ is a sum of restricted even period polynomials in $x_1,x_2$ multiplied by monomials in $x_3,\ldots,x_r$. More precisely, one can decompose
\begin{align}
    \mathbf W_{N,r}=\bigoplus_{\substack{1<n<N\\n\text{ even}}}\mathbf W_{n,2}\otimes \mathbf V_{N-n,r-2} \label{eq:decomposition}\;,
\end{align}
where $\mathbf W_{n,2}$ is the space of restricted even period polynomials of degree $n-2$. Since $\mathbf W_{n,2}$ is isomorphic to the space $\mathcal S_n$ of cusp forms of weight~$n$ by Eichler-Shimura correspondence (see \cite{zagier}), \eqref{eq:decomposition} leads to the following dimension formula.
\begin{lem}[Tasaka {\cite[equation 3.10]{tasakaPublished}}]\label{lem:wnreq}
For all $r\geq 2$,
\begin{equation*}
    \sum_{N>0}\dim_{\mathbb Q}\mathbf W_{N,r}\cdot x^N=  \mathbb O(x)^{r-2} \mathbb S(x)\;.
\end{equation*}
\end{lem}

\subsection{Ihara action and the matrices $E_{N,r}$ and $C_{N,r}$}
We use Tasaka's notation (from \cite{tasakaPublished}) for the polynomial representation of the Ihara action defined by Brown \cite[Section 6]{Brown}. Let
\begin{align*}
   \ihara\colon\mathbb Q[x_1]\otimes\mathbb Q[x_2,\ldots,x_r]&\longrightarrow\mathbb Q[x_1,\ldots,x_r] \\
   f\otimes g&\longmapsto f\ihara g\;,
\end{align*}
where $f\ihara g$ denotes the polynomial
\begin{multline*}
    (f\ihara g)(x_1,\ldots,x_r)\coloneqq f(x_1)g(x_2,\ldots,x_r)+\sum_{i=1}^{r-1}\Bigl(f(x_{i+1}-x_i)g(x_1,\ldots,\hat x_{i+1},\ldots,x_r)\\
   -(-1)^{\deg f}f(x_i-x_{i+1})g(x_1,\ldots,\hat x_i,\ldots,x_r)\Bigr)\;.
\end{multline*} 
(the hats are to indicate, that $x_{i+1}$ and $x_i$ resp. are omitted in the above expression).

For integers 
$m_1,\ldots,m_r,n_1,\cdots,n_r \geq 1$, let furthermore the integer $e\faul{m_1,\ldots,m_r}{n_1,\ldots,n_r}$ denote the coefficient of $x_1^{n_1-1}\cdots x_r^{n_r-1}$ in $x_1^{m_1-1} \ihara \left(x_1^{m_2-1}\cdots x_{r-1}^{m_r-1}\right)$, i.\,e.
\begin{align}\label{eq:e}
    x_1^{m_1-1} \ihara \left(x_1^{m_2-1}\cdots x_{r-1}^{m_r-1}\right) = \sum_{\substack{n_1 + \cdots + n_r = m_1 + \cdots + m_r \\ n_1, \cdots, n_r \geq 1}} e\faul{m_1,\ldots,m_r}{n_1,\ldots,n_r}x_1^{n_1-1}\cdots x_r^{n_r-1}\;.
\end{align}
Note that $e\faul{m_1,\ldots,m_r}{n_1,\ldots,n_r}=0$ if $m_1+\cdots+m_r\not=n_1+\cdots+n_r$.
\begin{rem}
One can explicitly compute the integers $e\faul{m_1,\ldots,m_r}{n_1,\cdots,n_r}$ by the following formula: (\cite[Lemma 3.1]{tasakaPublished})
\begin{multline*}
    e\faul{m_1,\ldots,m_r}{n_1,\ldots,n_r}=\delta\faul{m_1,\ldots,m_r}{n_1\ldots,n_r}+\sum_{i=1}^{r-1}\delta\faul{\hat m_1,m_2,\ldots,m_i,\hat m_{i+1},m_{i+2},\ldots,m_r}{n_1,\ldots,n_{i-1},\hat n_i,\hat n_{i+1},n_{i+2},\ldots,n_r}\\
    \cdot\left((-1)^{n_i}\binom{m_1-1}{n_i-1}+(-1)^{m_1-n_{i+1}}\binom{m_1-1}{n_{i+1}-1}\right)
\end{multline*}
(again, the hats are to indicate that $m_1,m_{i+1},n_i,n_{i+1}$ are omitted), where
\begin{equation*}
    \delta\faul{m_1,\ldots,m_{s}}{n_1,\ldots,n_{s}} \coloneqq  
    \begin{cases}
        1\quad\text{if } m_i=n_i \text{ for all }i\in\{1,\ldots,s\} \\
        0\quad\text{else}
    \end{cases}
\end{equation*}
denotes the usual Kronecker delta.
\end{rem}
\begin{dfn}\label{def:enrq}
Let $N,r$ be positive integers.
\begin{itemize}
\item[$(\operatorname{i})$]We define the $|S_{N,r}|\times |S_{N,r}|$ matrix
\begin{equation*}
    E_{N,r}\coloneqq \left(e\faul{m_1,\ldots,m_r}{n_1,\ldots,n_r}\right)_{(m_1,\ldots,m_r),(n_1,\ldots,n_r)\in S_{N,r}}\;.
\end{equation*}
\item[$(\operatorname{ii})$]For integers $r\geq j\geq 2$ we also define the $|S_{N,r}|\times |S_{N,r}|$ matrix
\begin{equation*}
    E_{N,r}^{(j)}\coloneqq \left(\delta\faul{m_1,\ldots,m_{r-j}}{n_1,\ldots,n_{r-j}}e\faul{m_{r-j+1},\ldots,m_r}{n_{r-j+1},\ldots,n_r} \right)_{(m_1,\ldots,m_r),(n_1,\ldots,n_r)\in S_{N,r}}\;.
\end{equation*}
\end{itemize}
\end{dfn}
\begin{dfn}[{\cite[Definition 2.3 and Proposition 3.3]{tasakaPublished}}]\label{def:cnr}
    The $|S_{N,r}|\times |S_{N,r}|$ matrix $C_{N,r}$ is defined as
    \begin{equation*}
        C_{N,r}\coloneqq E_{N,r}^{(2)}\cdot E_{N,r}^{(3)}\cdots E_{N,r}^{(r-1)}\cdot E_{N,r}\;.
    \end{equation*}
\end{dfn}

\section{Known Results}\label{sec:known results}

Recall the map $\pi\colon\mathbf V_{N,r}\rightarrow\Vect_{N,r}$ (equation~\eqref{eq:natiso}). Theorem~\ref{thm:Schneps} due to Baumard and Schneps \cite{schneps} establishes a connection between the left kernel of the matrix $E_{N,2}$ and the space $\mathbf W_{N,2}$ of restricted even period polynomials. This connection was further investigated by Tasaka \cite{tasakaPublished}, relating $\mathbf W_{N,r}$ and the left kernel of $E_{N,r}$ for arbitrary $r\geq2$.
\begin{thm}[Baumard-Schneps {\cite[Proposition 3.2]{schneps}}]\label{thm:Schneps}
For each integer $N>0$ we have
\begin{equation*}
    \pi\left(\mathbf W_{N,2}\right)=\ker\tr E_{N,2}\;.
\end{equation*}
\end{thm}
\begin{thm}[Tasaka {\cite[Theorem 3.6]{tasakaPublished}}, {\cite{tasakaCorrigendum}}]\label{thm:injection}
    Let $r\geq2$ be a positive integer and $F_{N,r}=E_{N,r}-\id_{\Vect_{N,r}}$. Then, the following $\mathbb Q$-linear map is well-defined:
    \begin{align}\label{eq:TasakasFail}
    \begin{split}
        \mathbf W_{N,r}&\longrightarrow \ker\tr E_{N,r}\\
        P(x_1,\ldots,x_r)&\longmapsto\pi(P)F_{N,r}\;.
    \end{split}
    \end{align}
\end{thm}
\begin{con}[Tasaka {\cite[Section 3.3]{tasakaPublished}}]\label{con:isomorphism}For all $r\geq2$, the map described in Theorem~\ref{thm:injection} is an isomorphism.
\end{con}
\begin{rem}\label{rem:isor2}
As far as the authors are aware, only the case $r=2$ is known, which is an immediate consequence of Theorem~\ref{thm:Schneps}. However, assuming injectivity of the morphisms \eqref{eq:TasakasFail} one has the following relation.
\end{rem}
\begin{cor}[Tasaka {\cite[Corollary 3.7]{tasakaPublished}}]\label{cor:enrineq}For all $r\geq2$,
    \begin{equation*}
        \sum_{N>0}\dim_{\mathbb Q}\ker\tr E_{N,r}\cdot x^N\geq\mathbb O(x)^{r-2}\mathbb S(x)\;.
    \end{equation*}
\end{cor}

\section{Main Tools}\label{sec:main tools}

\subsection{Decompositions of $E_{N,r}^{(j)}$}

We use the following decomposition lemma:
\begin{lem}\label{lem:blockdia}
Let $2\leq j\leq r-1$ and arrange the indices $(m_1,\ldots,m_r),(n_1,\ldots,n_r)\in S_{N,r}$ of $E_{N,r}^{(j)}$ in lexicographically decreasing order. Then, the matrix $E_{N,r}^{(j)}$ has block diagonal structure
\begin{equation*}
    E_{N,r}^{(j)}=\diag\left(E_{3r-3,r-1}^{(j)},E_{3r-1,r-1}^{(j)},\ldots,E_{N-3,r-1}^{(j)}\right)\;.
\end{equation*}
\end{lem}
\begin{proof}
    This follows directly from Definition~\ref{def:enrq}.
\end{proof}
\begin{cor}\label{cor:blockdia}
    We have
    \begin{equation*}
        E_{N,r}^{(2)}E_{N,r}^{(3)}\cdots E_{N,r}^{(r-1)}=\diag\left(C_{3r-3,r-1},C_{3r-1,r-1},\ldots,C_{N-3,r-1}\right)\;.
    \end{equation*}
\end{cor}
\begin{proof}
    Multiplying the block diagonal representations of $E_{N,r}^{(2)},E_{N,r}^{(3)},\ldots,E_{N,r}^{(r-1)}$ block by block together with Definition~\ref{def:cnr} yields the desired result.
\end{proof}
\begin{cor} \label{cor:enrjocnr}
    For all $r\geq 3$,
    \begin{equation*}
        \sum_{N>0}\dim_{\mathbb Q}\ker\left(E_{N,r}^{(2)}\cdots E_{N,r}^{(r-1)}\right)\cdot x^N=\mathbb O(x)\sum_{N>0}\dim_{\mathbb Q}\ker C_{N,r-1}\cdot x^N\;.
    \end{equation*}
\end{cor}
\begin{proof}
    According to Corollary~\ref{cor:blockdia}, the matrix $E_{N,r}^{(2)}\cdots E_{N,r}^{(r-1)}$ has block diagonal structure, the blocks being $C_{3r-3,r-1},C_{3r-1,r-1},\ldots,C_{N-3,r-1}$. Hence, 
    \begin{align*}
        \sum_{N>0}\dim_{\mathbb Q}\ker\left(E_{N,r}^{(2)}\cdots E_{N,r}^{(r-1)}\right)\cdot x^N &= \sum_{N>0} \left( \sum_{k=3r-3}^{N-3}\dim_{\mathbb Q}\ker C_{k,r-1} \right) \cdot x^N \\
        &=\sum_{N>0}\dim_{\mathbb Q}\ker C_{N,r-1}\left(x^{N+3}+x^{N+5}+x^{N+7}+\cdots\right)\\
        &=\mathbb O(x)\sum_{N>0}\dim_{\mathbb Q}\ker C_{N,r-1}\cdot x^N\;,
    \end{align*}
    thus proving the assertion.
\end{proof}

\subsection{Connection to polynomials}
Motivated by Theorem~\ref{thm:injection}, we interpret the right action of the matrices $E_{N,r}^{(2)},\ldots,E_{N,r}^{(r-1)},E_{N,r}^{(r)}=E_{N,r}$ on $\Vect_{N,r}$ as endomorphisms of the polynomial space $\mathbf V_{N,r}$. Having established this, we will prove Theorems \ref{thm:case3} and \ref{thm:case4} from a polynomial point of view.
\begin{dfn}\label{def:phij}
    The \emph{restricted totally even part} of a polynomial $Q(x_1,\ldots,x_r)\in\mathbf V_{N,r}$ is the sum of all of its monomials, in which each exponent of $x_1,\ldots,x_r$ is even and at least 2. Let $r\geq j$. We define the $\mathbb Q$-linear map
    \begin{equation*}
        \phi j\colon\mathbf V_{N,r}\longrightarrow\mathbf V_{N,r}\;,
    \end{equation*}
    which maps each polynomial $Q(x_1,\ldots,x_r)\in \mathbf V_{N,r}$ to the restricted totally even part of
    \begin{multline*}
        Q(x_1,\ldots,x_r)+\sum_{i=r-j+1}^{r-1}\Bigl(Q(x_1,\ldots,x_{r-j},x_{i+1}-x_i,x_{r-j+1},\ldots,\hat x_{i+1},\ldots,x_r)\\ -Q(x_1,\ldots,x_{r-j},x_{i+1}-x_i,x_{r-j+1},\ldots,\hat x_i,\ldots,x_r)\Bigr)\;.
    \end{multline*}
\end{dfn}
\begin{rem}
Note that $\phi 1\equiv\id_{\mathbf V_{N,r}}$. 
\end{rem}
The following lemma shows that the map $\phi j$ corresponds to the right action of the matrix $E_{N,r}^{(j)}$ on $\Vect_{N,r}$ via the isomorphism $\pi$.

\begin{lem}\label{lem:phij}
Let $r\geq j$. Then, for each polynomial $Q\in\mathbf V_{N,r}$,
\begin{equation*}
    \pi\left(\phi j(Q)\right)=\pi(Q)E_{N,r}^{(j)}\;.
\end{equation*}
or equivalently, the following diagram commutes:
\begin{center}
    \begin{tikzpicture}
        \node (a) at (0,2) {$\mathbf V_{N,r}$};
        \node (b) at (3,2) {$\mathbf V_{N,r}$};
        \node (c) at (0,0) {$\Vect_{N,r}$};
        \node (d) at (3,0) {$\Vect_{N,r}$};
        \node (e) at (0,-0.5) {\small$v$};
        \node (f) at (3,-0.5) {\small$v\cdot E_{N,r}^{(j)}$};
        \node (c1) at (e) {\phantom{$\Vect_{N,r}$}};
        \node (d1) at (f) {\phantom{$\Vect_{N,r}$}};
        \draw[->] (a) -- (b) node[pos=0.5,above]{\small$\phi j$};
        \draw[->] (c) -- (d);
        \draw[->] (a) -- (c) node[pos=0.5,sloped,above=-2pt]{$\sim$} node[pos=0.5,left]{\small$\pi$};
        \draw[->] (b) -- (d) node[pos=0.5,sloped,above=-2pt]{$\sim$} node[pos=0.5,left]{\small$\pi$};
        \draw[|->] (c1) -- (d1);
    \end{tikzpicture}
\end{center}

\end{lem}
\begin{proof}
We proceed by induction on $r$. Let $r=j$ and
\begin{equation*}
    Q(x_1,\ldots,x_j)=\sum_{(n_1,\ldots,n_j)\in S_{N,j}}q_{n_1,\ldots,n_j}x_1^{n_1-1}\cdots x_r^{n_j-1}\;.
\end{equation*}Then, $E_{N,j}^{(j)}=E_{N,j}$ and thus
\begin{equation*}
    \pi(Q)E_{N,j}=\left(\sum_{(m_1,\ldots,m_j)\in S_{N,j}}q_{n_1,\ldots,n_j}e\faul{m_1,\ldots,m_j}{n_1,\ldots,n_j}\right)_{(n_1,\ldots,n_j)\in S_{N,j}}\;.
\end{equation*}
By \eqref{eq:e} and linearity of the Ihara action $\ihara$, the row vector on the right-hand side corresponds to $\pi$ applied to the restricted totally even part of the polynomial
\begin{align}\label{eq:phiihara}
    \sum_{(n_1,\ldots,n_j)\in S_{N,j}}q_{n_1,\ldots,n_j}x_1^{n_1-1}\ihara\left(x_1^{n_2-1}\cdots x_{j-1}^{n_j-1}\right)\;.
\end{align}
On the other hand, plugging $r=j$ into Definition~\ref{def:phij} yields that $\phi[j]j(Q(x_1,\ldots,x_j))$ corresponds to the restricted totally even part of some polynomial, which by definition of the Ihara action $\ihara$ coincides with the polynomial defined in \eqref{eq:phiihara}. Thus, the claim holds for $r=j$.

Now suppose that $r\geq j+1$ and the claim is proven for all smaller $r$. Let us decompose
\begin{equation*}
    Q(x_1,\ldots,x_r)=\sum_{\substack{n_1=3\\n_1\text{ odd}}}^{N-(3r-3)}x_1^{n_1-1}\cdot Q_{N-n_1}(x_2,\ldots,x_r)\;,
\end{equation*}where the $Q_k$ are restricted totally even homogeneous polynomials in $r-1$ variables. In particular, $Q_k\in\mathbf V_{k,r-1}$ for all $k$. Arrange the indices of $\pi(Q)$ in lexicographically decreasing order. Then, by grouping consecutive entries, $\pi(Q)$ is the list-like concatenation of $\pi(Q_{3r-3}),\ldots,\pi(Q_{N-3})$, which we denote by
\begin{equation*}
    \pi(Q)=\big(\pi(Q_{3r-3}),\pi(Q_{3r-1}),\ldots,\pi(Q_{N-3})\big)\;.
\end{equation*}
Since we have lexicographically decreasing order of indices, the block diagonal structure of $E_{N,r}^{(j)}$ stated in Lemma~\ref{lem:blockdia} yields
\begin{align*}
    \pi(Q)E_{N,r}^{(j)}&=\left(\pi(Q_{3r-3})E_{3r-3,r-1}^{(j)},\pi(Q_{3r-1})E_{3r-1,r-1}^{(j)},\ldots,\pi(Q_{N-3})E_{N-3,r-1}^{(j)}\right)\\
    &=\left(\pi\left(\phi[r-1]j(Q_{3r-3})\right),\pi\left(\phi[r-1]j(Q_{3r-1})\right),\ldots,\pi\left(\phi[r-1]j(Q_{N-3})\right)\right)\\
    &=\pi\left(\phi j(Q)\right)
\end{align*}by linearity of $\phi j$ and the induction hypothesis. This shows the assertion.
\end{proof}

\begin{cor}\label{cor:phiEiso}
For all $r\geq2$,
\begin{equation*}
\im \tr{\left(E_{N,r}^{(2)}\cdots E_{N,r}^{(r-1)}\right)}\cap\ker\tr E_{N,r}\cong\ker\phi r\cap\im\left(\phi {r-1}\circ\cdots\circ\phi 2\right)\;.
\end{equation*}
\end{cor}
\begin{proof} By the previous Lemma~\ref{lem:phij}, the following diagram commutes:
\begin{center}
    \begin{tikzpicture}
        \node (a) at (0,2) {$\mathbf V_{N,r}$};
        \node (b) at (3,2) {$\mathbf V_{N,r}$};        
        \node (c) at (6,2) {$\cdots$};
        \node (d) at (9,2) {$\mathbf V_{N,r}$};
        \node (e) at (12,2) {$\mathbf V_{N,r}$};
        \node (A) at (0,0) {$\Vect_{N,r}$};
        \node (B) at (3,0) {$\Vect_{N,r}$};
        \node (C) at (6,0) {$\cdots$};
        \node (D) at (9,0) {$\Vect_{N,r}$};
        \node (E) at (12,0) {$\Vect_{N,r}$};
        \foreach \p/\q in{a/A,b/B,d/D,e/E}{
        \draw[->] (\p) -- (\q) node[pos=0.5,sloped,above=-2pt]{$\sim$} node[pos=0.5,left]{\small$\pi$};
        }
        \foreach \p/\q/\j in{a/b/2,b/c/3,c/d/r-1,d/e/r}{
        \draw[->] (\p) -- (\q) node[pos=0.5,above]{\small $\phi{\j}$};
        }
        \foreach \p/\q/\j in{A/B/2,B/C/3,C/D/r-1,D/E/r}{
        \draw[->] (\p) -- (\q) node[pos=0.5,above]{\small ${}\cdot E_{N,r}^{(\j)}$};
        }
    \end{tikzpicture}
\end{center}
From this, we have $\im\tr{\left(E_{N,r}^{(2)}\cdots E_{N,r}^{(r-1)}\right)}\cong\im\left(\phi {r-1}\circ\cdots\circ\phi 2\right)$ and $\ker\tr E_{N,r}\cong\ker\phi r$. Thereby, the claim is established.
\end{proof}

\begin{lem}\label{lem:kerphi}Let $j\leq r-1$. Then,
    \begin{equation*}
        \ker\phi j\cong\bigoplus_{n<N}\mathbf V_{N-n,r-j}\otimes\ker E_{n,j}\;.
    \end{equation*}
\end{lem}
\begin{proof}
Let $Q\in\mathbf V_{N,r}$. We may decompose
\begin{equation*}
    Q(x_1,\ldots,x_r)=\sum_{n<N}\ \sum_{(n_1,\ldots,n_{r-j})\in S_{N-n,r-j}}x_1^{n_1-1}\cdots x_{r-j}^{n_{r-j}-1}R_{n_1,\ldots,n_{r-j}}(x_{r-j+1},\ldots,x_r)\;,
\end{equation*}
where $R_{n_1,\ldots,n_{r-j}}\in\mathbf V_{n,j}$ is a restricted totally even homogeneous polynomial. Note that we have $Q\in\ker\phi j$ if and only if $\phi[j]j(R_n)=0$ holds for each $R_n$ in the above decomposition. By Lemma~\ref{lem:phij}, $\phi[j]j(R_n)=0$ if and only if $\pi(R_n)\in\ker E_{n,j}$. Now, the assertion is immediate.
\end{proof}
\begin{cor}\label{cor:phijres}
    Let $2\leq j\leq r-2$. The restricted map
    \begin{equation*}
        \phii j{N,r}\colon\mathbf W_{N,r}\longrightarrow\mathbf W_{N,r}
    \end{equation*}
    is well-defined and satisfies
    \begin{equation*}
        \ker\phii j{N,r}\cong\bigoplus_{n<N}\mathbf W_{N-n,r-j}\otimes\ker E_{n,j}\;.
    \end{equation*}
\end{cor}
\begin{proof}
Since $j\leq r-2$, for each $Q\in\mathbf V_{N,r}$ the map $Q(x_1,\ldots,x_r)\mapsto\phi j(Q)$ does not interfere with $x_1$ or $x_2$ and thus not with the defining property of $\mathbf W_{N,r}$. Hence, $\phii j{N,r}$ is well-defined. The second assertion is done just like in the previous Lemma~\ref{lem:kerphi}.
\end{proof}
\begin{lem}\label{lem:fnr}
Let $r\geq3$. For all $P\in\mathbf W_{N,r}$,
\begin{equation*}
    \pi(-P)E_{N,r}^{(r-1)}=\pi(P)F_{N,r}\;.
\end{equation*}
\end{lem}
\begin{proof}Recall that by Lemma~\ref{lem:phij}, 
\begin{align*}
    \pi(-P)E_{N,r}^{(r-1)}&=\pi\left( \phi {r-1}\big(-P(x_1,\ldots,x_r)\big)\right)\\
    &=\begin{multlined}[t]
    \pi\bigg(-P(x_1,\ldots,x_r)+\sum_{i=2}^{r-1}\Big(-P(x_1,x_{i+1}-x_i,x_2,\ldots, \hat x_{i+1},\ldots,x_r)\\
    +P(x_1,x_{i+1}-x_i,x_2,\ldots, \hat x_i,\ldots,x_r)\Big)\bigg)
    \end{multlined}\\
    &=\begin{multlined}[t]
    \pi\bigg(-P(x_1,\ldots,x_r)+\sum_{i=2}^{r-1}\Big(P(x_{i+1}-x_i,x_1,\ldots, \hat x_{i+1},\ldots,x_r)\\
    -P(x_{i+1}-x_i,x_1,\ldots, \hat x_i,\ldots,x_r)\Big)\bigg)\;,
    \end{multlined}
\end{align*}
since $-P$ is antisymmetric with respect to $x_1\leftrightarrow x_2$. In the same way we compute
\begin{align*}
    \pi(P)F_{N,r}&=\pi(P)\left(E_{N,r}-\id_{\Vect_{N,r}}\right)=\pi\big(\phi r(P(x_1,\ldots,x_r))\big)-\pi(P)\\
    &=\begin{multlined}[t]
    \pi\bigg(P(x_1,\ldots,x_r)+\sum_{i=1}^{r-1}\Big(P(x_{i+1}-x_i,x_1,\ldots, \hat x_{i+1},\ldots,x_r)\\
    -P(x_{i+1}-x_i,x_1,\ldots, \hat x_i,\ldots,x_r)\Big)\bigg)-\pi(P)\;.
    \end{multlined}
\end{align*}
Now the desired result follows from
\begin{equation*}
    P(x_1,x_2,\ldots,x_r)+P(x_2-x_1,x_1,x_3,\ldots,x_r)-P(x_2-x_1,x_2,\ldots,x_r)=0\;,
\end{equation*}
since $P$ is in $\mathbf W_{N,r}$.
\end{proof}
\begin{cor}\label{cor:kerimineq}
Assume that the map from Theorem~\ref{thm:injection} is injective. Then, for all $r\geq 3$,
\begin{equation*}
    \dim_{\mathbb Q}\left(\im\tr E_{N,r}^{(r-1)}\cap\ker\tr E_{N,r}\right)\geq\dim_{\mathbb Q}\mathbf W_{N,r}\;.
\end{equation*}
\end{cor}
\begin{proof}
This is immediate by the previous Lemma~\ref{lem:fnr}.
\end{proof}
\begin{lem}\label{lem:imphi}
For all $r\geq3$,
\begin{equation*}
    \im\left(\phi{r-1}\circ\phii{r-2}{N,r}\circ\cdots\circ\phii{2}{N,r}\right)\subseteq\ker \phi{r}\cap\im\left(\phi{r-1}\circ\cdots\circ\phi{2}\right)\;.
\end{equation*}
\end{lem}
\begin{proof}
We may replace the right-hand side by just $\ker\phi r$. Note that by Corollary~\ref{cor:phijres} the composition of restricted $\phii j{N,r}$ on the left-hand side is well-defined. Moreover, each $Q\in\im\left(\phi{r-1}\circ\phii{r-2}{N,r}\circ\cdots\circ\phii{2}{N,r}\right)$ can be represented as $Q=\phi{r-1}(P)$ for some $P\in\mathbf W_{N,r}$ and thus $Q\in\ker\phi r$ according to Lemma~\ref{lem:fnr} and Theorem~\ref{thm:injection}.
\end{proof}
Similar to Conjecture~\ref{con:isomorphism} we expect a stronger result to be true, which is stated in the following conjecture due to Claire Glanois:
\begin{con}\label{con:imiso}
For all $r\geq 3$,
\begin{equation*}
 \im\left(\phi{r-1}\circ\phii{r-2}{N,r}\circ\cdots\circ\phii{2}{N,r}\right)=\ker \phi r\cap\im\left(\phi{r-1}\circ\cdots\circ\phi{2}\right)\;.
\end{equation*}
\end{con}
\begin{rem}
Note that intersecting $\ker\phi r\cap\im\left(\phi{r-1}\circ\cdots\circ\phi{2}\right)$, Conjecture~\ref{con:imiso} does not need the injectivity from Conjecture~\ref{con:isomorphism}. However, we haven't been able to derive Conjecture~\ref{con:imiso} from Conjecture~\ref{con:isomorphism}, so it is not necessarily weaker.
\end{rem}

\section{Main Results}\label{sec:main results}
Throughout this section we will assume that the map from Theorem~\ref{thm:injection} is injective, i.e. the injectivity part of Conjecture~\ref{con:isomorphism} is true. This was also the precondition for Tasaka's original proof of Theorem~\ref{thm:case4}.

\subsection{Proof of Theorem~\ref{thm:case3}, case $r=3$.} 
By Corollary~\ref{cor:enrjocnr}, Remark~\ref{rem:isor2} and the fact that $E_{N,2}=C_{N,2}$ we obtain 
\begin{align}\label{eq:case3ineq1}
    \sum_{N>0}\dim_{\mathbb Q}\ker E_{N,3}^{(2)}\cdot x^N=\mathbb O(x)\sum_{N>0}\dim_{\mathbb Q}\ker C_{N,2}\cdot x^N=\mathbb O(x)\mathbb S(x)\;.
\end{align}
We use Corollary~\ref{cor:kerimineq} and Lemma~\ref{lem:wnreq} to obtain
\begin{align}\label{eq:case3ineq2}
    \sum_{N>0}\dim_{\mathbb Q}\left(\im\tr E_{N,3}^{(2)} \cap \ker\tr E_{N,3}\right)\cdot x^N \geq \sum_{N>0}\dim_{\mathbb Q}\mathbf W_{N,3}\cdot x^N
    = \mathbb O(x) \mathbb S(x)\;.
\end{align}
Now observe that since $C_{N,3}=E_{N,3}^{(2)}E_{N,3}$, we have
\begin{equation*}
    \dim_{\mathbb Q}\ker C_{N,3}=\dim_{\mathbb Q}\ker\tr E_{N,3}^{(2)}+\dim_{\mathbb Q}\left(\im\tr E_{N,3}^{(2)} \cap \ker\tr E_{N,3}\right)\;.
\end{equation*}
By \eqref{eq:case3ineq1} and \eqref{eq:case3ineq2}, the assertion is proven.\qed

\subsection{Proof of Theorem~\ref{thm:case4}, case $r=4$.}Since $C_{N,4}=E_{N,4}^{(2)}E_{N,4}^{(3)}E_{N,4}$, we may split $\dim_{\mathbb Q}\ker C_{N,4}$ into
\begin{equation*}
\dim_{\mathbb Q}\ker C_{N,4}=\dim_{\mathbb Q}\ker\tr{\left( E_{N,4}^{(2)}E_{N,4}^{(3)}\right)}+\dim_{\mathbb Q}\left(\im\tr {\left( E_{N,4}^{(2)}E_{N,4}^{(3)}\right)}
\cap \ker\tr E_{N,4}\right)\;.
\end{equation*}
The two summands on the right-hand side are treated separately. For the first one, by Corollary~\ref{cor:enrjocnr} and Theorem~\ref{thm:case3} one has
\begin{align}\label{eq:case4ineq1}
    \sum_{N>0}\dim_{\mathbb Q}\ker\tr{\left( E_{N,4}^{(2)}E_{N,4}^{(3)}\right)}\cdot x^N\geq2\mathbb O(x)^2\mathbb S(x)\;.
\end{align}
For the second one, we use Corollary~\ref{cor:phiEiso} and Lemma~\ref{lem:imphi} to obtain
\begin{align*}
    \dim_{\mathbb Q}\left(\im\tr{\left( E_{N,4}^{(2)}E_{N,4}^{(3)}\right)} \cap \ker\tr E_{N,4}\right)&\geq\dim_{\mathbb Q}\im\left(\phi[4]3\circ\phii[4]2{N,4}\right)\\
    &=\dim_{\mathbb Q}\im\phii[4]2{N,4}\;,
\end{align*}
since we assume $\phi[4]3$ to be injective on $\mathbf W_{N,r}$ according to Conjecture~\ref{con:isomorphism}. According to Corollary~\ref{cor:phijres} and Theorem~\ref{thm:Schneps},
\begin{equation*}
    \ker\phii[4]2{N,4}\cong\bigoplus_{n<N}\mathbf W_{N-n,2}\otimes\ker E_{n,2}\cong\bigoplus_{n<N}\mathbf W_{N-n,2}\otimes\mathbf W_{n,2}\;.
\end{equation*}
Now, by $\dim_{\mathbb Q}\im\phii[4]2{N,4}=\dim_{\mathbb Q}\mathbf W_{N,4}-\dim_{\mathbb Q}\ker\phii[4]2{N,4}$ we obtain
\begin{align}\label{eq:case4ineq2}
    \sum_{N>0}\dim_{\mathbb Q}\im\phii[4]2{N,4}\cdot x^N=\mathbb O(x)^2\mathbb S(x)-\mathbb S(x)^2\;.
\end{align}
Combining \eqref{eq:case4ineq1} and \eqref{eq:case4ineq2}, the proof is finished.\qed

\subsection{Proof of Theorem~\ref{thm:case5}.}
In addition to the injectivity of \eqref{eq:TasakasFail}, we now assume Conjecture~\ref{con:isomorphism} is true in the case $r=3$, i.e. 
\begin{align}
     \sum_{N>0}\dim_{\mathbb Q}\ker\tr E_{N,3}\cdot x^N=\mathbb O(x)\mathbb S(x)\label{eq:case5eq1}
\end{align}
by Corollary~\ref{cor:enrineq}. Our goal is to prove the lower bound
\begin{align}
     \sum_{N>0}\dim_{\mathbb Q}\ker C_{N,5}\cdot x^N\geq 4 \mathbb O(x)^3\mathbb S(x)-3 \mathbb O(x)\mathbb S(x)^2\;,\label{eq:case5}
\end{align}
which as an equality would be the exact value predicted by Conjecture~\ref{con:Brown}.
Again we use the decomposition $C_{N,5}=E_{N,5}^{(2)}E_{N,5}^{(3)}E_{N,5}^{(4)}E_{N,5}$ to split $\dim_{\mathbb Q}\ker C_{N,5}$ into
\begin{equation*}
\dim_{\mathbb Q}\ker C_{N,5}=\dim_{\mathbb Q}\ker\tr{\left( E_{N,5}^{(2)}E_{N,5}^{(3)}E_{N,5}^{(4)}\right)}+\dim_{\mathbb Q}\left(\im\tr{ \left( E_{N,5}^{(2)}E_{N,5}^{(3)}E_{N,5}^{(4)}\right)} \cap \ker\tr E_{N,5}\right)\;.
\end{equation*}
Applying Corollary~\ref{cor:enrjocnr} and Theorem~\ref{thm:case4} to the first summand on the right-hand side, we obtain
\begin{align}\label{eq:case5ineq1}
    \sum_{N>0}\dim_{\mathbb Q}\ker E_{N,5}^{(2)}E_{N,5}^{(3)}E_{N,5}^{(4)}\cdot x^N\geq3\mathbb O(x)^3\mathbb S(x)-\mathbb O(x)\mathbb S(x)^2\;.
\end{align}
Again, for the second summand Corollary~\ref{cor:phiEiso} and Lemma~\ref{lem:imphi} yield
\begin{align*}
\begin{split}
    \dim_{\mathbb Q}\left(\im\tr{\left( E_{N,5}^{(2)}E_{N,5}^{(3)}E_{N,5}^{(4)}\right)} \cap \ker\tr E_{N,5}\right)&\geq\dim_{\mathbb Q}\im\left(\phi[5]4\circ\phii[5]3{N,5}\circ\phii[5]2{N,5}\right)\\
    &=\dim_{\mathbb Q}\im\left(\phii[5]3{N,5}\circ\phii[5]2{N,5}\right)\;,
\end{split}
\end{align*}
since $\phi[5]4$ is injective on $\mathbf W_{N,r}$ by our assumption. According to Corollary~\ref{cor:phijres} and Theorem~\ref{thm:Schneps},
\begin{equation*}
    \ker\phii[5]2{N,5}\cong\bigoplus_{n<N}\mathbf W_{N-n,3}\otimes\ker E_{n,2}\cong\bigoplus_{n<N}\mathbf W_{N-n,3}\otimes\mathbf W_{n,2}
\end{equation*}
and by \eqref{eq:case5eq1},
\begin{equation*}
    \ker\phii[5]3{N,5}\cong\bigoplus_{n<N}\mathbf W_{N-n,2}\otimes\ker E_{n,3}\cong\bigoplus_{n<N}\mathbf W_{N-n,2}\otimes\mathbf W_{n,3}\;.
\end{equation*}
Now, by
\begin{equation*}
\dim_{\mathbb Q}\im\left(\phii[5]3{N,5}\circ\phii[5]2{N,5}\right)\geq\dim_{\mathbb Q}\mathbf W_{N,5}-\dim_{\mathbb Q}\ker\phii[5]2{N,5}-\dim_{\mathbb Q}\ker\phii[5]3{N,5}
\end{equation*}
we arrive at
\begin{align}\label{eq:case5ineq2}
    \sum_{N>0}\dim_{\mathbb Q}\im\left(\phii[5]3{N,5}\circ\phii[5]2{N,5}\right)\cdot x^N\geq\mathbb O(x)^3\mathbb S(x)-2\mathbb O(x)\mathbb S(x)^2\;.
\end{align}
Combining \eqref{eq:case5ineq1} and \eqref{eq:case5ineq2} yields the desired result.\qed

\subsection{A recursive approach to the general case $r\geq2$}
In this section, we show that one can recursively derive the exact value of $\dim_{\mathbb Q}\ker C_{N,r}$ from Conjecture~\ref{con:imiso}. Let us fix some notations:
\begin{dfn}
For $r\geq2$, let us define the formal series
\begin{align}
    B_r(x)&\coloneqq \sum_{N>0}\dim_{\mathbb Q}\im\left(\phi{r-1}\circ\phii{r-2}{N,r}\circ\cdots\circ\phii2{N,r}\right)\cdot x^N\tag i\\
    T_r(x)&\coloneqq \sum_{N>0}\dim_{\mathbb Q}\ker C_{N,r}\cdot x^N\tag{ii}\;.
\end{align}
We set $T_0(x),T_1(x)\coloneqq 0$.
\end{dfn}

The main observation is the following lemma:
\begin{lem}\label{lem:recursion}
Assume that Conjecture~\ref{con:imiso} is true and that the map from Theorem~\ref{thm:injection} is injective. Then, for $r\geq 3$ the following recursion holds:
\begin{equation*}
B_r(x)=\mathbb O(x)^{r-2}\mathbb S(x)-\sum_{j=2}^{r-2}\mathbb O(x)^{r-j-2}\mathbb S(x)B_j(x)\;.
\end{equation*}
\end{lem}
\begin{proof}
We have
\begin{multline*}
\dim_{\mathbb Q}\im\left(\phi{r-1}\circ\phii{r-2}{N,r}\circ\cdots\circ\phii2{N,r}\right)\\
=\begin{aligned}[t]
\dim_{\mathbb Q}\mathbf W_{N,r}&-
\sum_{j=2}^{r-2}\dim_{\mathbb Q}\ker\phii j{N,r}\cap\im\left(\phii{j-1}{N,r}\circ\cdots\circ\phii2{N,r}\right)\\
&-\dim_{\mathbb Q}\ker\phi{r-1}\cap\im\left(\phii{r-2}{N,r}\circ\cdots\circ\phii2{N,r}\right)\;.
\end{aligned}
\end{multline*}
Since we assume $\phi {r-1}$ to be injective on $\mathbf W_{N,r}$, the last summand on the right-hand side vanishes. Let $2\leq j\leq r-2$. As the restriction to $\mathbf W_{N,r}$ only affects $x_1$ and $x_2$, whereas $\phi j$ acts on $x_{r-j+1},\ldots,x_r$, we obtain
\begin{multline*}
\ker\phii j{N,r}\cap\im\left(\phii{j-1}{N,r}\circ\cdots\circ\phii2{N,r}\right)\\
\begin{aligned}[t]
&=\bigoplus_{n<N}\mathbf W_{N-n,r-j}\otimes\left(\ker\phi[j]j\cap\im\left(\phi[j]{j-1}\circ\cdots\circ\phi[j]2\right)\right)\\
&=\bigoplus_{n<N}\mathbf W_{N-n,r-j}\otimes\im\left(\phi[j]{j-1}\circ\phii[j]{j-2}{n,j}\circ\cdots\circ\phii[j]2{n,j}\right)\;,
\end{aligned}
\end{multline*}
where the last equality follows from Conjecture~\ref{con:imiso}. Hence, if we denote 
\begin{equation*}
a_{N,r}=\dim_{\mathbb Q}\im\left(\phi{r-1}\circ\phii{r-2}{N,r}\circ\cdots\circ\phii2{N,r}\right)\;,
\end{equation*}we obtain the recursion 
\begin{align}\label{eq:casercauchy}
a_{N,r}=\dim_{\mathbb Q}\mathbf W_{N,r}-\sum_{j=2}^{r-2}\sum_{n<N}\dim_{\mathbb Q}\mathbf W_{N-n,r-j}\cdot a_{n,j}\;.
\end{align}
By Lemma~\ref{lem:wnreq} and the convolution formula for multiplying formal series, equation~\eqref{eq:casercauchy} establishes the claim.
\end{proof}

\begin{thm}\label{thm:recursion}
Upon Conjecture~\ref{con:imiso} and the injectivity of \eqref{eq:TasakasFail}, for all $r\geq3$ the following recursion is satisfied:
\begin{equation*}
T_r(x)=\mathbb O(x)T_{r-1}(x)-\mathbb S(x)T_{r-2}(x)+\mathbb O(x)^{r-2}\mathbb S(x)\;.
\end{equation*}
\end{thm}
\begin{proof}
As we assume Conjecture~\ref{con:imiso}, we get from Definition~\ref{def:cnr} and Corollary~\ref{cor:phiEiso}
\begin{align*}
\dim_{\mathbb Q}\ker C_{N,r}&=
\begin{multlined}[t]
\dim_{\mathbb Q}\ker\tr{\left( E_{N,r}^{(2)}\cdots E_{N,r}^{(r-1)}\right)}\\+\dim_{\mathbb Q}\left(\im\tr{\left( E_{N,r}^{(2)}\cdots E_{N,r}^{(r-1)}\right)} \cap \ker\tr E_{N,r}\right)
\end{multlined}\\
&=\begin{multlined}[t]
\dim_{\mathbb Q}\ker\tr{\left( E_{N,r}^{(2)}\cdots E_{N,r}^{(r-1)}\right)}\\+\dim_{\mathbb Q}\im\left(\phi{r-1}\circ\phii{r-2}{N,r}\circ\cdots\circ\phii2{N,r}\right)
\end{multlined}
\end{align*}
and thus, by Corollary~\ref{cor:enrjocnr},
\begin{align}
T_r(x)=\mathbb O(x)T_{r-1}(x)+B_r(x)\;.
\end{align}
Using Lemma~\ref{lem:recursion}, we obtain
\begin{align*}
T_r(x)&=\mathbb O(x)T_{r-1}(x)+\mathbb O(x)^{r-2}\mathbb S(x)-\sum_{j=2}^{r-2}\mathbb O(x)^{r-j-2}\mathbb S(x)\big(T_j(x)-\mathbb O(x)T_{j-1}(x)\big)\\
&=\mathbb O(x)T_{r-1}(x)+\mathbb O(x)^{r-2}\mathbb S(x)-\mathbb S(x)T_{r-2}(x)+\mathbb O(x)^{r-3}\mathbb S(x)T_1(x)\\
&=\mathbb O(x)T_{r-1}(x)-\mathbb S(x)T_{r-2}(x)+\mathbb O(x)^{r-2}\mathbb S(x)\;,
\end{align*}where by definition $T_1(x)=0$. The conclusion follows.
\end{proof}

Note that by our choice of $T_0(x)$ and $T_1(x)$, Theorem~\ref{thm:recursion} remains true for $r=2$ since we know from \cite{schneps} that $T_2(x)=\mathbb S(x)$. Under the assumption of Conjecture~\ref{con:imiso} and injectivity in \eqref{eq:TasakasFail}, we are now ready to prove that the generating series of $\rank C_{N,r}$ equals the explicit series $\frac{1}{1-\mathbb O(x)y+\mathbb S(x)y^2}$ as was claimed in Conjecture~\ref{con:Brown}. This (under the same assumptions though) proves the motivic version of Conjecture~\ref{con:Brown} (i.e. with $\mathcal Z_{N,r}^\odd$ replaced by $\mathcal H_{N,r}^\odd$). 

Let $R_r(x)=\mathbb O(x)^r-T_r(x)$ and note that by Theorem~\ref{thm:recursion}
\begin{equation*}
    R_r(x)=\mathbb O(x)R_{r-1}(x)-\mathbb S(x)R_{r-2}(x)
\end{equation*}for all $r\geq2$. Hence,
\begin{align*}
    \left(1-\mathbb O(x)y+\mathbb S(x)y^2\right)\sum_{r\geq0}R_r(x)y^r&=
    \begin{aligned}[t]
        &\sum_{r\geq2}\big(R_r(x)-\mathbb O(x)R_{r-1}(x)+\mathbb S(x)R_{r-2}(x)\big)y^r\\
        &+R_0(x)+R_1(x)y-\mathbb O(x)R_0(x)y
    \end{aligned}\\
    &=R_0(x)+\mathbb O(x)y-\mathbb O(x)y\\
    &=1
\end{align*}and thus
\begin{equation*}
    1+\sum_{N,r>0}\rank C_{N,r}\cdot x^Ny^r=\sum_{r\geq0}R_r(x)y^r=\frac1{1-\mathbb O(x)y+\mathbb S(x)y^2}\;,
\end{equation*}
which is the desired result (Theorem~\ref{thm:caser}).

\newpage
\printbibliography

\end{document}